\documentclass[12pt,a4paper]{amsart}
\usepackage{amssymb}
\usepackage{amsmath}
\allowdisplaybreaks[2] 
\newcount\minutes \newcount\hours
\hours=\time
\divide\hours 60
\minutes=\hours
\multiply\minutes -60
\advance\minutes \time
\newcommand{\klockan}{\the\hours:{\ifnum\minutes<10 0\fi}\the\minutes}
\newcommand{\tid}{\today\ \klockan}
\newcommand{\prtid}{\smash{\raise 10mm \hbox{\LaTeX ed \tid}}}
\renewcommand{\prtid}{}
\makeatletter
\def\sectionmark#1{} 
\def\subsectionmark#1{}
\newcommand{\sectnr}{\ifnum \c@secnumdepth >\z@
                 \thesection.\hskip 1em\relax \fi}
\def\@evenhead{\footnotesize\rm\thepage\hfil\leftmark\hfil\llap{\prtid}}
\def\@oddhead{\footnotesize\rm\rlap{\prtid}\hfil\rightmark\hfil\thepage}
\def\tableofcontents{\section*{Contents} 
 \@starttoc{toc}}
\makeatother
\makeatletter
\let\Thebibliography=\thebibliography
\renewcommand{\thebibliography}[1]{\def\@mkboth##1##2{}\Thebibliography{#1}
\addcontentsline{toc}{section}{References}
\frenchspacing 
\setlength{\@topsep}{0pt}
\setlength{\itemsep}{0pt}%
\setlength{\parskip}{0pt plus 2pt}%
}
\makeatother
\makeatletter
\def\mdots@{\mathinner.\nonscript\!.%
 \ifx\next,.\else\ifx\next;.\else\ifx\next..\else
 \nonscript\!\mathinner.\fi\fi\fi}
\let\ldots\mdots@
\let\cdots\mdots@
\let\dotso\mdots@
\let\dotsb\mdots@
\let\dotsm\mdots@
\let\dotsc\mdots@
\def\vdots{\vbox{\baselineskip2.8\p@ \lineskiplimit\z@
    \kern6\p@\hbox{.}\hbox{.}\hbox{.}\kern3\p@}}
\def\ddots{\mathinner{\mkern1mu\raise8.6\p@\vbox{\kern7\p@\hbox{.}}%
    \raise5.8\p@\hbox{.}\raise3\p@\hbox{.}\mkern1mu}}
\makeatother
\makeatletter
\let\Enumerate=\enumerate
\renewcommand{\enumerate}{\Enumerate%
\setlength{\@topsep}{0pt}
\setlength{\itemsep}{0pt}%
\setlength{\parskip}{0pt plus 1pt}%
\renewcommand{\theenumi}{\textup{(\alph{enumi})}}%
\renewcommand{\labelenumi}{\theenumi}%
}
\let\endEnumerate=\endenumerate
\renewcommand{\endenumerate}{\endEnumerate\unskip}
\makeatother
\makeatletter
\def\@seccntformat#1{\csname the#1\endcsname.\quad}
\makeatother
\newcommand{\authortitle}[2]{\author{#1}\title{#2}\markboth{#1}{#2}}
\newcommand{\art}[6]{{\sc #1, \rm #2, \it #3 \bf #4 \rm (#5), \mbox{#6}.}}
\newcommand{\book}[3]{{\sc #1, \it #2, \rm #3.}}
\newcommand{\AND}{{\rm and }}
\RequirePackage{amsthm}
\newtheoremstyle{descriptive}%
  {\topsep}   
  {\topsep}   
  {\rmfamily} 
  {}          
  {\bfseries} 
  {.}         
  { }         
  {}          
\newtheoremstyle{propositional}%
  {\topsep}   
  {\topsep}   
  {\itshape}  
  {}          
  {\bfseries} 
  {.}         
  { }         
  {}          
\theoremstyle{propositional}
\newtheorem{thm}{Theorem}[section]

\newtheorem{lemma}[thm]{Lemma} 
\newtheorem{cor}[thm]{Corollary}
\theoremstyle{descriptive}
\newtheorem{deff}[thm]{Definition}

\newtheorem{remark}[thm]{Remark}
\makeatletter
\renewenvironment{proof}[1][\proofname]{\par
  \pushQED{\qed}%
  \normalfont
  \trivlist
  \item[\hskip\labelsep
        \itshape
    #1\@addpunct{.}]\ignorespaces
}{%
  \popQED\endtrivlist\@endpefalse
}
\makeatother
\newdimen\extrawidth
\def\iintlim#1#2{\setbox0\hbox{$\scriptstyle#1$}%
        \setbox1\hbox{$\scriptstyle#2$}%
        \extrawidth=\wd1 \advance\extrawidth-\wd0
        \ifdim\extrawidth<0pt \extrawidth=0pt\fi%
        \int_{#1\kern\extrawidth \kern .5em}^{#2\kern -\wd1} \kern -.5em%
}
\newcommand{\setm}{\setminus}
\def\vint_#1{\mathchoice%
          {\mathop{\kern 0.2em\vrule width 0.6em height 0.69678ex depth -0.58065ex
                  \kern -0.8em \intop}\nolimits_{\kern -0.4em#1}}%
          {\mathop{\kern 0.1em\vrule width 0.5em height 0.69678ex depth -0.60387ex
                  \kern -0.6em \intop}\nolimits_{#1}}%
          {\mathop{\kern 0.1em\vrule width 0.5em height 0.69678ex depth -0.60387ex
                  \kern -0.6em \intop}\nolimits_{#1}}%
          {\mathop{\kern 0.1em\vrule width 0.5em height 0.69678ex depth -0.60387ex
                  \kern -0.6em \intop}\nolimits_{#1}}}
\def\vintslides_#1{\mathchoice%
          {\mathop{\kern 0.1em\vrule width 0.5em height 0.697ex depth -0.581ex
                  \kern -0.6em \intop}\nolimits_{\kern -0.4em#1}}%
          {\mathop{\kern 0.1em\vrule width 0.3em height 0.697ex depth -0.604ex
                  \kern -0.4em \intop}\nolimits_{#1}}%
          {\mathop{\kern 0.1em\vrule width 0.3em height 0.697ex depth -0.604ex
                  \kern -0.4em \intop}\nolimits_{#1}}%
          {\mathop{\kern 0.1em\vrule width 0.3em height 0.697ex depth -0.604ex
                  \kern -0.4em \intop}\nolimits_{#1}}}
\DeclareMathOperator{\diam}{diam}

\DeclareMathOperator{\Lip}{Lip}

\DeclareMathOperator{\spt}{supp}

\DeclareMathOperator{\capc}{cap}
\DeclareMathOperator{\Capc}{Cap}

\newcommand{\loc}{_{\rm loc}}
{\catcode`p =12 \catcode`t =12 \gdef\eeaa#1pt{#1}}      
\def\accentadjtext#1{\setbox0\hbox{$#1$}\kern   
                \expandafter\eeaa\the\fontdimen1\textfont1 \ht0 }
\def\accentadjscript#1{\setbox0\hbox{$#1$}\kern 
                \expandafter\eeaa\the\fontdimen1\scriptfont1 \ht0 }
\def\accentadjscriptscript#1{\setbox0\hbox{$#1$}\kern   
                \expandafter\eeaa\the\fontdimen1\scriptscriptfont1 \ht0 }
\def\accentadjtextback#1{\setbox0\hbox{$#1$}\kern       
                -\expandafter\eeaa\the\fontdimen1\textfont1 \ht0 }
\def\accentadjscriptback#1{\setbox0\hbox{$#1$}\kern     
                -\expandafter\eeaa\the\fontdimen1\scriptfont1 \ht0 }
\def\accentadjscriptscriptback#1{\setbox0\hbox{$#1$}\kern 
                -\expandafter\eeaa\the\fontdimen1\scriptscriptfont1 \ht0 }

\newcommand{\dmu}{d\mu}

\newcommand{\Om}{\Omega}

\renewcommand{\phi}{\varphi}
\newcommand{\p}{{$p\mspace{1mu}$}}
\newcommand{\R}{\mathbb{R}}
\newcommand{\N}{\mathbf{N}}
\newcommand{\limminus}{{\mathchoice{\raise.17ex\hbox{$\scriptstyle -$}}
                {\raise.17ex\hbox{$\scriptstyle -$}}
                {\raise.1ex\hbox{$\scriptscriptstyle -$}}
                {\scriptscriptstyle -}}}
\newcommand{\limplus}{{\mathchoice{\raise.17ex\hbox{$\scriptstyle +$}}
                {\raise.17ex\hbox{$\scriptstyle +$}}
                {\raise.1ex\hbox{$\scriptscriptstyle +$}}
                {\scriptscriptstyle +}}}
\newcommand{\limpm}{{\mathchoice{\raise.17ex\hbox{$\scriptstyle \pm$}}
                {\raise.17ex\hbox{$\scriptstyle \pm$}}
                {\raise.16ex\hbox{$\scriptscriptstyle \pm$}}
                {\scriptscriptstyle \pm}}}
\newcommand{\Np}{N^{1,p}}

\newcommand{\Hn}{\mathbb H^n}
\newcommand{\g}{\gamma}

\makeatletter
\newcommand{\setcurrentlabel}[1]{\def\@currentlabel{#1}}
\makeatother
\numberwithin{equation}{section}
\begin{document}

\authortitle{}
{Sharp capacitary estimates for rings in metric spaces}

\author{Nicola Garofalo}
\address{Department of Mathematics\\Purdue University \\
West Lafayette, IN 47907, USA} \email[Nicola
Garofalo]{garofalo@math.purdue.edu}
\thanks{First author supported in part by NSF Grant DMS-0701001}

\author{Niko Marola}
\address{Department of Mathematics and Systems Analysis \\
Helsinki University of Technology \\
P.O. Box 1100 FI-02015 TKK \\
Finland} \email[Niko Marola]{niko.marola@tkk.fi}
\thanks{Second author supported by the Academy of Finland and Emil
Aaltosen s\"a\"ati\"o}

\date{}

\subjclass[2000]{Primary: 31B15, 31C45; Secondary: 31C15}

\keywords{Capacity, doubling measure, Green function, Newtonian
space, \p-harmonic, Poincar\'e inequality, singular function,
Sobolev space}

\maketitle

\noindent{\small {\bf Abstract}. We establish sharp estimates for
the $p$-capacity of metric rings with unrelated radii in metric
measure spaces equipped with a doubling measure and supporting a
Poincar\'e inequality. These estimates play an essential role in the
study of the local behavior of \p-harmonic Green's functions. }

\section{Introduction}

In this paper we establish sharp capacitary estimates for the metric
rings with unrelated radii in a locally doubling metric measure
space supporting a local $(1,p)$-Poincar\`e inequality. A motivation
for pursuing these estimates comes from the study of the asymptotic
behavior of
 \p-harmonic Green's functions in this geometric setting. Similarly
to the classical case (for the latter the reader should see
\cite{LSW}, \cite{Ser1} and \cite{Ser2}), capacitary estimates play
a crucial role in studying the local behavior of such singular
functions. For this aspect we refer the reader to the forthcoming
paper by Danielli and the authors~\cite{DaGaMa}.

Perhaps the most important model of a metric space with a rich
non-Euclidean geometry is the Heisenberg group $\Hn$, whose
underlying manifold is $\mathbb C^n \times \mathbb R$ with the group law
$(z,t)\circ(z',t') = (z+z',t+t' - \frac{1}{2} \mathcal Im (z
\overline z'))$. Kor\'anyi and Reimann~\cite{KoRe} first computed
explicitly the $Q$-capacity of a metric ring in $\Hn$. Here $Q =
2n+2$ indicates the homogeneous dimension of $\mathbb H^n$ attached
to the non-isotropic group dilations $\delta_\lambda(z,t) = (\lambda
z,\lambda^2 t)$. Their method makes use of a suitable choice of
``polar'' coordinates in the group.

The Heisenberg group is the prototype of a general class of
nilpotent stratified Lie groups, nowadays known as Carnot groups. In
this more general context, Heinonen and Holopainen~\cite{HeHo}
proved sharp estimates for the $Q$-capacity of a ring. Again, here
$Q$ indicates the homogeneous dimension attached to the
non-isotropic dilations associated with the grading of the Lie
algebra.

In the paper \cite{CaDaGa} Capogna, Danielli and the first named
author established sharp $p$-capacitary estimates, for the range
$1<p<\infty$, for Carnot--Carath\'eodory rings associated with a
system of vector fields of H\"ormander type. In particular they
proved that for a ring centered at a point $x$ the \p-capacity of
the ring itself changes drastically depending on whether $1<p<Q(x)$,
$p=Q(x)$ or $p>Q(x)$. Here, $Q(x)$ is the pointwise dimension at
$x$, and such number in general differs from the so-called local
homogeneous dimension associated with a fixed compact set containing
$x$. This unsettling phenomenon is not present, for example, in the
analysis of Carnot groups since in that case $Q(x)\equiv Q$, where
$Q$ is the above mentioned homogeneous dimension of the group.

In ~\cite{KiMa96},\cite{KiMaNov} Kinnunen and Martio developed a
capacity theory based on the definition of Sobolev functions on
metric spaces. They also provided sharp upper bounds for the
capacity of a ball.

The results in the present paper encompass all previous ones and
extend them. For the relevant geometric setting of this paper we
refer the reader to Section \ref{S:prelim}.

Here, we confine ourselves to mention that a fundamental example of
the spaces included in this paper is obtained by endowing a
connected Riemannian manifold $M$ with the Carath\'eodory metric
$d$ associated with a given subbundle of the tangent bundle,
see \cite{Ca}. If such subbundle generates the tangent space at
every point, then thanks to the theorem of Chow \cite{Chow} and
Rashevsky \cite{Ra} $(M,d)$ is a metric space. Such metric spaces
are known as sub-Riemannian or Carnot-Carath\'eodory (CC) spaces. By
the fundamental works of Rothschild and Stein \cite{RS}, Nagel,
Stein and Wainger \cite{NSW}, and of Jerison \cite{J}, every CC
space is locally doubling, and it locally satisfies a
$(p,p)$-Poincar\`e inequality for any $1\leq p<\infty$. Another
basic example is provided by a Riemannian manifold $(M^n,g)$ with
nonnegative Ricci tensor. In such case thanks to the Bishop
comparison theorem the doubling condition holds globally, see e.g.
\cite{Ch}, whereas the $(1,1)$-Poincar\`e inequality was proved by
Buser \cite{Bu}. An interesting example to which our results apply
and that does not fall in any of the two previously mentioned
categories is the space of two infinite closed cones
$X=\{(x_1,\ldots, x_n)\in\R^n:\ x_1^2+\ldots +x_{n-1}^2\leq x_n^2\}$
equipped with the Euclidean metric of $\R^n$ and with the Lebesgue
measure. This space is Ahlfors regular, and it is shown in Haj\l
asz--Koskela~\cite[Example 4.2]{HaKo} that a $(1,p)$-Poincar\'e
inequality holds in $X$ if and only if $p>n$. Another example is
obtained by gluing two copies of closed $n$-balls $\{x\in \R^n:\
|x|\leq 1\}$, $n\geq 3$, along a line segment. In this way one
obtains an Ahlfors regular space that supports a $(1,p)$-Poincare
inequality for $p>n-1$.

A thorough overview of analysis on metric spaces can be found in
Heinonen~\cite{heinonen}. One should also consult
Semmes~\cite{Semmes} and David and Semmes~\cite{DaSem}.

The present note is organized as follows. In Section \ref{S:prelim}
we list our main assumptions and gather the necessary background
material. In Section \ref{S:ce} we establish sharp capacitary
estimates for spherical rings with unrelated radii. Section
\ref{S:singfunctions} closes the paper with a small remark on the
existence of \p-harmonic Green's functions. In the setting of metric
measure spaces Holopainen and Shanmugalingam \cite{HoSha}
constructed a \p-harmonic Green's function, called a singular
function there, having most of the characteristics of the
fundamental solution of the Laplace operator. See also
Holopainen~\cite{Holopainen}.

\subsection*{Acknowledgements}

This paper was completed while the second author was visiting Purdue
University in 2007--2008. He wishes to thank the Department of
Mathematics for the hospitality and several of its faculty for
fruitful conversations.

\section{Preliminaries}
\label{S:prelim}

We begin by introducing our main assumptions on the metric space $X$
and on the measure $\mu$.

\subsection{General assumptions} \label{assumptions}

Throughout the paper $X=(X,d,\mu)$ is a locally compact metric space
endowed with a metric $d$ and a positive Borel regular measure $\mu$
such that $0<\mu(B(x,r)) <\infty$ for all balls $B(x,r):=\{y\in X:
d(y,x)<r\}$ in~$X$. We assume that for every compact set $K\subset
X$ there exist constants $C_K \geq 1$, $R_K>0$ and $\tau_K \ge 1$,
such that for any $x\in K$ and every $0<2r\leq R_K$, one has:
\begin{itemize}
\item[(i)] the closed balls $\overline B(x,r)=\{y\in X:d(y,x)\leq r\}$ are compact;
\item[(ii)] (local doubling condition) $\mu(B(x,2r)) \le C_K \mu(B(x,r))$;
\item[(iii)] (local weak $(1,p_0)$-Poincar\'e inequality) there exists $1<p_0<\infty$ such that for all $u\in N^{1,p_0}(B(x,\tau_Kr))$ and all weak
upper gradients $g_u$ of $u$
\begin{equation*} \label{PI-ineq}
        \vint_{B(x,r)} |u-u_{B(x,r)}| \,\dmu
        \le C_K r \Big( \vint_{B(x,\tau_K r)} g_u^{p_0} \,\dmu \Big)^{1/p_0},
\end{equation*}
where $ u_{B(x,r)} :=\vint_{B(x,r)}u \, d\mu
:=\int_{B(x,r)} u\, d\mu/\mu(B(x,r))$. Given an open set
$\Om\subseteq X$, and $1<p<\infty$, the notation $N^{1,p}(\Om)$
indicates the $p$-Newtonian space on $\Om$ defined below.
\end{itemize}

Hereafter, the constants $C_K, R_K$ and $\tau_K$ will be referred to
as the \emph{local parameters} of $K$. We also say that a constant $C$
depends on the local doubling constant of $K$ if $C$ depends on $C_K$.

The above assumptions encompass, e.g., all Riemannian manifolds with
Ric $\geq 0$, but they also include all Carnot--Carath\'eodory
spaces, and therefore, in particular, all Carnot groups. For a
detailed discussion of these facts  we refer the reader to the paper
by  Garofalo--Nhieu~\cite{GaNhi}. In the case of
Carnot--Carath\'eodory spaces, recall that if the Lie algebra
generating vector fields grow at infinity faster than linearly, then
the compactness of metric balls of large radii may fail in general.
Consider for instance in $\R$ the smooth vector field of H\"ormander
type $X_1=(1+x^2)\frac{d}{dx}$. Some direct calculations prove that
the distance relative to $X_1$ is given by
$d(x,y)=|\arctan(x)-\arctan(y)|$, and therefore, if $r\geq \pi/2$,
we have $B(0,r)=\R$.

\subsection{Local doubling property}
We note that assumption (ii) implies that for every compact set
$K\subset X$ with local parameters $C_K$ and $R_K$, for any $x\in K$
and every $0<r\leq R_K$, one has for $1\leq \lambda \leq R_K/r$,
\begin{equation}\label{dc}
\mu(B(x,\lambda r)) \leq C\lambda^Q\mu(B(x,r)), \end{equation}
 where
$Q=\log_2C_K$, and the constant $C$ depends only on the local doubling
constant $C_K$. The exponent $Q$ serves as a local dimension of the
doubling measure $\mu$ restricted to the compact set $K$. In
addition to such local dimension, for $x\in X$ we define the
\emph{pointwise dimension} $Q(x)$ by
\begin{multline*}
Q(x) = \sup\{q > 0:\ \exists C>0\ \textrm{ such that } \\
\lambda^q\mu(B(x,r)) \leq C\mu(B(x,\lambda r)),
\textrm{ for all } \lambda\geq1, 0<r<\infty\}.
\end{multline*}

The inequality \eqref{dc} readily implies that $Q(x) \leq Q$ for
every $x\in K$. Moreover, it follows that
\begin{equation} \label{lowerbound}
\lambda^{Q(x)}\mu(B(x,r)) \leq C\mu(B(x,\lambda r))
\end{equation}
for any $x \in K$, $0<r\leq R_K$ and $1\leq \lambda \leq R_K/r$,
and the constant $C$ depends on the local doubling constant $C_K$.
Furthermore, for all $0<r\leq R_K$ and $x\in K$
\begin{equation} \label{bounds}
C_1r^Q\leq \frac{\mu(B(x,r))}{\mu(B(x,R_K))} \leq C_2r^{Q(x)},
\end{equation}
where $C_1=C(K,C_K)$ and $C_2= C(x,K,C_K)$.

For more on doubling measures, see, e.g. Heinonen~\cite{heinonen}
and the references therein.

\subsection{Upper gradients}
A path is a continuous mapping from a compact interval, and we say that a
nonnegative Borel function $g$ on $X$ is an \emph{upper gradient}
of an extended real valued function $f$
on $X$ if for all rectifiable paths $\gamma$
joining points $x$ and $y$ in $X$ we have
\begin{equation} \label{ug-cond}
|f(x)-f(y)|\le \int_\gamma g\,ds.
\end{equation}
whenever both $f(x)$ and $f(y)$ are finite, and $\int_\g g\,
ds=\infty $ otherwise. See Cheeger~\cite{Cheeger} and
Shanmugalingam~\cite{Sh-rev} for a detailed discussion of upper
gradients.

If $g$ is a nonnegative measurable function on $X$
and if (\ref{ug-cond}) holds for \p-almost every path,
then $g$ is a \emph{weak upper gradient} of~$f$.
By saying that (\ref{ug-cond}) holds for \p-almost every path
we mean that it fails only for a path family with zero \p-modulus
(see, for example, \cite{Sh-rev}).

A function $f:X\to\R$ is \emph{Lipschitz}, denoted by
$f \in \Lip(X)$, if there exists a constant
$L \geq 0$ such
that $|f(x)-f(y)| \leq Ld(x,y)$ for every $x,y\in X$. The
\emph{upper pointwise Lipschitz constant} of $f$ at $x$ defined by
\[
\Lip f(x) = \limsup_{r \to 0}\sup_{y\in B(x,r)}\frac{|f(y)-f(x)|}{r}
\]
is an upper gradient of $f$. We note that for $c\in\R$, $\Lip f(x) =
0$ for $\mu$-a.e. $x\in \{y\in X: f(y)=c\}$.

If $f$ has an upper gradient in $L^p(X)$, then
it has a \emph{minimal weak upper gradient} $g_f \in L^p(X)$
in the sense
 that for every weak upper gradient $g \in L^p(X)$ of $f$,
$g_f \le g$ $\mu$-almost everywhere (a.e.),
see Corollary~3.7 in Shanmugalingam~\cite{Sh-harm}, and
Lemma~2.3 in J.~Bj\"orn~\cite{Bj} for the pointwise characterization
of $g_f$.

Thanks to the results in Cheeger~\cite{Cheeger}, if $X$ satisfies
assumptions (ii) and (iii), then for $f\in\Lip(X)$ one has $g_f(x) =
\Lip f(x)$ for $\mu$-a.e. $x\in X$.

We recall the following version of the chain rule.

\medskip

\begin{lemma} \label{lemma:chainrule}
Let $u\in\Lip(X)$ and $f:\R\to\R$ be absolutely continuous and differentiable.
Then
\[
g_{f\circ u}(x) \leq |(f'\circ u)(x)|\Lip u(x)
\]
for $\mu$-almost every $x\in X$.
\end{lemma}

\subsection{Capacity}
Let $\Om \subset X$ be open and $E \subset \Om$ a Borel set.
The \emph{relative \p-capacity} of $E$ with respect to $\Om$ is the number
\begin{equation*}
  \Capc_p (E,\Om) =\inf\int_\Om g_u^p\,d\mu,
\end{equation*}
where the infimum is taken over all functions $u \in \Np(X)$ such that
$u=1$ on $E$ and $u=0$ on $X\setminus\Om$. If such a function do
not exist, we set $\Capc_p (K,\Om)=\infty$. When $\Om=X$ we simply
write $\Capc_p(E)$.

Observe that if $E\subset \Om$ is compact the infimum above could be
taken over all functions
$u \in \Lip_0(\Om)=\{f \in \Lip(X):\ f=0 \textrm{ on } X\setminus\Om\}$
such that $u=1$ on $E$.

Suppose that $\Om \subset X$ is open. A \emph{condenser} is a
triple $(E,F;\Om)$, where $E,F \subset \Om$ are disjoint non-empty compact
sets. For $1\leq p < \infty$
the \p-\emph{capacity of a condenser} is the number
\[
\capc_p(E,F;\Om) = \inf\int_\Om g^p\, d\mu,
\]
where the infimum is taken over all \p-weak upper gradients $g$
of all functions $u$ in $\Om$ such that $u = 0$ on $E$, $u = 1$ on $F$,
and $0\leq u \leq 1$.

For other properties as well as equivalent definitions of
the capacity we refer to Kilpel\"ainen et al.~\cite{KiKiMa},
Kinnunen--Martio~\cite{KiMa96, KiMaNov}, and
Kallunki--Shanmugalingam~\cite{KaSh}. See also Gol'dshtein and
Troyanov~\cite{GoTro}.

\subsection{Newtonian spaces}

We define Sobolev spaces on the metric space
following Shanmugalingam~\cite{Sh-rev}. Let $\Om\subseteq X$ be
nonempty and open.
        Whenever $u\in L^p(\Om)$, let
$$
        \|u\|_{\Np(\Om)} = \biggl( \int_\Om |u|^p \, \dmu
                + \inf_g  \int_\Om g^p \, \dmu \biggr)^{1/p},
$$
where the infimum is taken over all weak upper gradients of $u$.
The \emph{Newtonian space} on $\Om$ is the quotient space
$$
        \Np (\Om) = \{u: \|u\|_{\Np(\Om)} <\infty \}/{\sim},
$$
where  $u \sim v$ if and only if $\|u-v\|_{\Np(\Om)}=0$.
The Newtonian space is a Banach space and a lattice, moreover,
Lipschitz functions are dense;
for the properties of Newtonian spaces we refer to \cite{Sh-rev}
and Bj\"orn et al.~\cite{BBS}.

To be able to compare the boundary values of Newtonian functions
we need a Newtonian space with zero boundary values.
Let $E$ be a measurable subset of $X$. The \emph{Newtonian space with zero
boundary values} is the space
\[
\Np_0(E)=\{u|_{E} : u \in \Np(X) \text{ and }
        u=0 \text{ on } X \setm E\}.
\]
The space $\Np_0(E)$ equipped with the norm inherited from $\Np(X)$
is a Banach space, see Theorem~4.4 in Shanmugalingam~\cite{Sh-harm}.

We say that $u$ belongs to the \emph{local Newtonian space}
$\Np\loc(\Omega)$ if $u\in \Np(\Om')$ for every open
$\Om'\Subset\Omega$ (or equivalently that $u\in \Np(E)$ for every measurable $E\Subset\Omega$).

\section{Capacitary estimates}\label{S:ce}

The aim of this section is to establish sharp capacity estimates for
metric rings with unrelated radii. We emphasize an interesting
feature of Theorems \ref{thm:below} and \ref{thm:above} that cannot
be observed, for example, in the setting of Carnot groups. That is
the dependence of the estimates on the center of the ring. This is a
consequence of the fact that in this generality $Q(x_0) \neq Q$
where $x_0 \in X$, see Section~\ref{S:prelim}. The results in this
section will play an essential role in the subsequent developments,
see the forthcoming paper by Danielli and the authors~\cite{DaGaMa}.

For now on, let $0< r < \frac1{10}\diam(X)$ and fix a ball
$B(x_0,r) \subset X$. We have
the following estimate.

\medskip

\begin{lemma} \label{lemma:pointwisest}
Let $u\in\Lip(X)$ such that $u =0$ in $X\setminus B(x_0,r)$. Then
\begin{equation}
|u(x)| \leq C\biggl(r^{p_0-1}\int_{B(x_0,r)}\frac{(\Lip u)^{p_0}(y)d(x,y)}{\mu(B(x,d(x,y)))}\,d\mu(y)\biggr)^{1/p_0},
\end{equation}
for all $x \in B(x_0,r)$.
\end{lemma}

\medskip

For the proof see, e.g.,
M\"akel\"ainen \cite{Makelainen}, Theorem 3.2 and Remark
3.3.

We are ready to prove sharp capacitary estimates for metric
rings with unrelated radii.

\medskip

\begin{thm}(Estimates from below) \label{thm:below}
Let $\Om \subset X$ be a bounded open set, $x_0 \in \Om$, and $Q(x_0)$ be
the pointwise dimension at $x_0$. Then there exists $R_0 = R_0(\Om)>0$
such that for any $0<r<R<R_0$
we have
\begin{align*}
& \Capc_{p_0}(\overline{B}(x_0,r),B(x_0,R)) \geq \\
& \left\{\begin{array}{ll}
C_1(1-\frac{r}{R})^{p_0(p_0-1)}\frac{\mu(B(x_0,r))}{r^{p_0}},\ \textrm{ if }\ 1<p_0<Q(x_0), \\
C_2(1-\frac{r}{R})^{Q(x_0)(Q(x_0)-1)}\biggl(\log\frac{R}{r}\biggr)^{1-Q(x_0)},
\ \textrm{ if }\ p_0=Q(x_0), \\
C_3(1-\frac{r}{R})^{p_0(p_0-1)}\biggl|(2R)^{\frac{p_0-Q(x_0)}{p_0-1}}-r^{\frac{p_0-Q(x_0)}{p_0-1}}\biggr|^{1-p_0},\ \textrm{ if }\ p_0>Q(x_0),
\end{array} \right.
\end{align*}
where
\begin{align*}
C_1 & = C\biggl(1-\frac1{2^{\frac{Q(x_0)-p_0}{p_0-1}}}\biggr)^{p_0-1}, \\
C_2 & =C\frac{\mu(B(x_0,r))}{r^{Q(x_0)}},\\
C_3 & =C\frac{\mu(B(x_0,r))}{r^{Q(x_0)}}\biggl(2^{\frac{p_0-Q(x_0)}{p_0-1}}-1\biggr)^{p_0-1},
\end{align*}
with $C > 0$ depending only on $p_0$ and the doubling constant
of $\Om$.
\end{thm}

\begin{proof}
Let $u \in \Lip(X)$ such that $u=0$ on $X\setminus B(x_0,R)$, $u=1$
in $B(x_0,r)$, and $0\leq u \leq 1$. Then by Lemma~\ref{lemma:pointwisest}
\begin{align*}
1 = |u(x_0)| & \leq C\biggl(R^{p_0-1}\int_{B(x_0,R)}\frac{(\Lip u)^{p_0}(y)d(x_0,y)}{\mu(B(x_0,d(x_0,y)))}\,d\mu(y)\biggr)^{1/p_0} \\
& \leq C\biggl(\left(\frac{R}{R-r}\right)^{p_0-1}\int_{B(x_0,R)}\frac{(\Lip u)(y)\,d(x_0,y)}{\mu(B(x_0,d(x_0,y)))}\,d\mu(y)\biggr)^{1/p_0} \\
& \leq C\left(\frac{R}{R-r}\right)^{1-1/p_0}\biggl(\int_{B(x_0,R)}(\Lip u)^{p_0}(y)\,d\mu(y)\biggr)^{1/p_0^2} {}{} \\
&{}{} \cdot\biggl(\int_{B(x_0,R)\setminus \overline{B}(x_0,r)}\frac{d(x_0,y)^{p'_0}}{\mu(B(x_0,d(x_0,y)))^{p'_0}}\,d\mu(y)\biggr)^{1/p'_0p_0},
\end{align*}
where $p'_0= p_0/(p_0-1)$. We choose $k_0 \in \N$ so that $2^{k_0}r
\leq R < 2^{k_0+1}r$. Then we get
\begin{align*}
\int_{B(x_0,R)\setminus \overline{B}(x_0,r)}&\frac{d(x_0,y)}{\mu(B(x_0,d(x_0,y)))}\,d\mu(y) \\
& \leq C\sum_{k=0}^{k_0}\int_{B(x_0,2^{k+1}r)\setminus \overline{B}(x_0,2^k r)}\frac{d(x_0,y)^{p'_0}}{\mu(B(x_0,d(x_0,y)))^{p'_0}}\,d\mu(y) \\
& \leq C\sum_{k=0}^{k_0}\frac{(2^kr)^{p'_0}}{\mu(B(x_0,2^kr))^{p'_0-1}} \\
& \leq C\frac{r^{p'_0}}{\mu(B(x_0,r))^{p'_0-1}}
\sum_{k=0}^{k_0}2^{k(p'_0-Q(x_0)(p'_0-1))}.
\end{align*}
If $1<p_0<Q(x_0)$, then $p'_0-Q(x_0)(p'_0-1)<0$, and we obtain
\begin{equation} \label{p<Q}
1 \leq C_1^{-1}\biggl(\frac{R}{R-r}\biggr)^{p_0(p_0-1)}\frac{r^{p_0}}{\mu(B(x_0,r))}\int_{B(x_0,R)}(\Lip u)^{p_0}\,d\mu.
\end{equation}
If $p_0=Q(x_0)$, then $p'_0-Q(x_0)(p'_0-1)=0$, and we find
\begin{equation} \label{p=Q}
1 \leq C_2^{-1}\biggl(\frac{R}{R-r}\biggr)^{Q(x_0)(p_0-1)}\frac{r^{Q(x_0)}}{\mu(B(x_0,r))}k_0^{p_0-1}\int_{B(x_0,R)}(\Lip u)^{p_0}\,d\mu.
\end{equation}
Finally, if $p_0 > Q(x_0)$, then $p'_0-Q(x_0)(p'_0-1)>0$, and we have
\begin{multline} \label{p>Q}
1 \leq C_3^{-1}\biggl(\frac{R}{R-r}\biggr)^{p_0(p_0-1)}\frac{r^{Q(x_0)}}{\mu(B(x_0,r))} \\
\cdot \biggl|(2R)^{\frac{p_0-Q(x_0)}{p_0-1}}-r^{\frac{p_0-Q(x_0)}{p_0-1}}\biggr|^{p_0-1}\int_{B(x_0,R)}(\Lip u)^{p_0}\,d\mu.
\end{multline}
Taking the infimum over all competing $u$'s in
\eqref{p<Q}--\eqref{p>Q} we reach the desired conclusion.
\end{proof}

\medskip

\begin{remark}
Observe that if $X$ supports the weak $(1,1)$-Poincar\'e inequality,
i.e. $p_0 =1$, these estimates reduce to the capacitary estimates, e.g.,
in Capogna et al.~\cite[Theorem 4.1]{CaDaGa}.
\end{remark}

\medskip

\begin{thm}(Estimates from above) \label{thm:above}
Let $\Om$, $x_0$, and $Q(x_0)$ be as in Theorem~\ref{thm:below}.
Then there exists $R_0 = R_0(\Om)>0$ such that for any $0<r<R<R_0$ we have
\begin{align*}
& \Capc_{p_0}(\overline{B}(x_0,r),B(x_0,R)) \\
& \leq \left\{\begin{array}{ll}
C_4\frac{\mu(B(x_0,r))}{r^{p_0}}, & \textrm{ if }\,
1<p_0<Q(x_0), \\
C_5\biggl(\log\frac{R}{r}\biggr)^{1-Q(x_0)}, & \textrm{ if }\, p_0=Q(x_0), \\
C_6\left|(2R)^{\frac{p_0-Q(x_0)}{p_0-1}}-r^{\frac{p_0-Q(x_0)}{p_0-1}}\right|^{1-p_0}, &
\textrm{ if }\, p_0>Q(x_0),
\end{array} \right.
\end{align*}
where $C_4$ is a positive constant depending only
on  $p_0$ and the local doubling constant of $\Om$, whereas
\[
C_5 = C\frac{\mu(B(x_0,r))}{r^{Q(x_0)}},
\]
with $C>0$ depending only on  $p_0$ and the local doubling constant
of $\Om$. Finally,
\[
C_6 =C\biggl(2^{\frac{p_0-Q(x_0)}{p_0-1}}-1\biggr)^{-1},
\]
with $C > 0$ depending on $p_0$, the local parameters of $\Om$, and
$\mu(B(x_0,R_0))$.
\end{thm}

\begin{proof}
For $i = 0,1$ and $p\neq Q(x_0)$, we define
\[
h(t) = \left\{\begin{array}{ll}
1, & \textrm{ if }\, 0\leq t \leq r, \\
\frac{t^{\frac{p_0-Q_i}{p_0-1}}-R^{\frac{p_0-Q_i}{p_0-1}}}{r^{\frac{p_0-Q_i}{p_0-1}}-R^{\frac{p_0-Q_i}{p_0-1}}}, & \textrm{ if }\, r\leq t \leq R, \\
0, & \textrm{ if }\, t \geq R,
\end{array} \right.
\]
where $Q_0 = Q$ and $Q_1 = Q(x_0)$. Note that $h\in L^{\infty}(\R)$,
$\spt(h') \subset [r,R]$, and that $h'\in L^{\infty}(\R)$, thus $h$
is a Lipschitz function. Let $u = h \circ d(x_0,y)$. By the chain
rule, see Lemma~\ref{lemma:chainrule}, we obtain for $\mu$-a.e.
\begin{align*}
g_u^{p_0} & \leq (|(h'\circ d(x_0,y))|\Lip d(x_0,y))^{p_0} \\
& = \biggl|\frac{p_0-Q_i}{p_0-1}\biggr|^{p_0}
\frac{d(x_0,y)^{\frac{(1-Q_i)p_0}{p_0-1}}}{\left|r^{\frac{p_0-Q_i}{p_0-1}}-
R^{\frac{p_0-Q_i}{p_0-1}}\right|^{p_0}}.
\end{align*}
Furthermore, we have that
\begin{align*}
\Capc_{p_0} & (\overline{B}(x_0,r),B(x_0,R)) \leq \int_{B(x_0,R)\setminus \overline{B}(x_0,r)}g_u^{p_0}\,d\mu \\
& \leq \sum_{k=0}^{k_0}\int_{B(x_0,2^{k+1}r)\setminus \overline{B}(x_0,2^kr)}g_u^{p_0}\,d\mu \\& \leq C\biggl|r^{\frac{p_0-Q_i}{p_0-1}}-R^{\frac{p_0-Q_i}{p_0-1}}\biggr|^{-p_0}\sum_{k=0}^{k_0}(2^kr)^{\frac{(1-Q_i)p_0}{p_0-1}}\mu(B(x_0,2^kr)),
\end{align*}
where $k_0\in\N$ is chosen so that $2^{k_0}r \leq R < 2^{k_0+1}r$.

At this point we need to make a distinction. If $1<p_0<Q(x_0)\leq Q$, then
we select $i=0$, and we have by the doubling property that
\begin{align*}
\Capc_{p_0} & (\overline{B}(x_0,r),B(x_0,R)) \\
& \leq C\biggl|r^{\frac{p_0-Q}{p_0-1}}-R^{\frac{p_0-Q}{p_0-1}}\biggr|^{-p_0}\mu(B(x_0,r))r^{\frac{(1-Q)p_0}{p_0-1}}\sum_{k=0}^{k_0}2^{k\frac{(p_0-Q)}{p_0-1}}
\\
& \leq C\biggl|1-\biggl(\frac{R}{r}\biggr)^{\frac{p_0-Q}{p_0-1}}\biggr|^{-p_0}\frac{\mu(B(x_0,r))}{r^{p_0}}.
\end{align*}
This completes the proof in the range $1<p_0<Q(x_0)$.

When $p_0>Q(x_0)$, from the second inequality in \eqref{bounds} it
follows
\[
\mu(B(x_0,2^kr)) \leq C2^{kQ(x_0)}r^{Q(x_0)},
\]
where the constant $C$ depends on $p_0$, the local doubling constant of
$\Om$, $\Om$, and $\mu(B(x_0,R_0))$.
Then we set $i=1$, and obtain
\begin{align*}
& \Capc_{p_0}(\overline{B}(x_0,r),B(x_0,R)) \\
& \leq C\biggl|r^{\frac{p_0-Q(x_0)}{p_0-1}}-R^{\frac{p_0-Q(x_0)}{p_0-1}}\biggr|^{-p_0}\mu(B(x_0,r))r^{Q(x_0)+\frac{(1-Q(x_0))p_0}{p_0-1}}\sum_{k=0}^{k_0}2^{k\frac{(p_0-Q(x_0))}{p_0-1}} \\
& \leq C(2^{\frac{p_0-Q(x_0)}{p_0-1}}-1)^{-1}\biggl|(2R)^{\frac{p_0-Q(x_0)}{p_0-1}}-r^{\frac{p_0-Q(x_0)}{p_0-1}}\biggr|^{1-p_0}.
\end{align*}
This end the proof in the range $p_0>Q(x_0)$.

When $p_0=Q(x_0)$ we set
\[
h(t) = \left\{\begin{array}{lr}
1, & \textrm{ if }\, 0\leq t \leq r, \\
\Big(\log \frac{R}{r}\Big)^{-1}\log\frac{R}{t}, & \textrm{ if }\, r\leq t \leq R, \\
0, & \textrm{ if }\, t \geq R,
\end{array} \right.
\]
As above, let $u = h \circ d(x_0,y)$, and Lemma~\ref{lemma:chainrule}
implies for $\mu$-a.e.
\[
g_u^{p_0} \leq \Big(\log\frac{R}{r}\Big)^{-p_0}\frac{1}{d(x_0,y)^{p_0}}.
\]
We have
\begin{align*}
\Capc_{p_0} & (\overline{B}(x_0,r),B(x_0,R)) \leq \int_{B(x_0,R)\setminus \overline{B}(x_0,r)}g_u^{p_0}\,d\mu \\
& \leq \sum_{k=0}^{k_0}\int_{B(x_0,2^{k+1}r)\setminus \overline{B}(x_0,2^kr)}g_u^{p_0}\,d\mu \\
& \leq C\Big(\log\frac{R}{r}\Big)^{-Q(x_0)}\sum_{k=0}^{k_0}(2^kr)^{-Q(x_0)}\mu(B(x_0,2^kr)) \\
& \leq C\frac{\mu(B(x_0,r))}{r^{Q(x_0)}}\Big(\log\frac{R}{r}\Big)^{-Q(x_0)},
\end{align*}
where the inequality \eqref{dc} was used and $k_0\in\N$ was chosen
so that $2^{k_0}r \leq R < 2^{k_0+1}r$. This completes the proof.
\end{proof}

\medskip

We have the following immediate corollary.

\medskip

\begin{cor} \label{corollary}
If $1< p_0 \leq Q(x_0)$, then we have
\begin{equation*} \label{eq:singleton}
\Capc_{p_0}(\{x_0\},\Om) = 0.
\end{equation*}
\end{cor}

\medskip

We close this section by stating for completeness
the following well-known
estimate for the conformal capacity.
In the setting of Carnot groups it was first proved by
Heinonen in \cite{Hei}.
For a discussion in metric spaces, see
Heinonen--Koskela~\cite[Theorem 3.6]{HeKo} and
Heinonen~\cite[Theorem 9.19]{heinonen}. In \cite{heinonen}
a weak $(1,1)$-Poincar\'e inequality is assumed. By obvious
modifications, however, the proof carries out in our setting as well.
We hence omit the proof.

\medskip

\begin{thm} \label{thm:CapAwayZero}
Suppose that $E$ and $F$ are connected closed subsets of $X$ such
that $F$ is unbounded and $F\cap\partial B(z,r) \neq \varnothing$,
and $E$ joins $z$ to $\partial B(z,r)$. Then there is a uniform
constant $C
> 0$, depending only on X, such that
\begin{equation*}
\capc_Q(E,F;X) \geq C > 0.
\end{equation*}
\end{thm}

\section{A remark on the existence of singular functions}
\label{S:singfunctions}

In this section we give a remark on the existence of
singular functions or \p$_0$-harmonic Green's
functions on relatively compact
domains $\Om \subset X$. The existence of singular
functions in metric space setting was proved in
Holopainen--Shanmugalingam~\cite{HoSha} in $Q$-regular metric
spaces (see below) supporting a local Poincar\'e inequality.

We start off by recalling the definition of \p$_0$-harmonic funcition on
metric spaces. Let $\Om\subset X$ be a domain. A
function $u \in N^{1,p_0}\loc(\Om)\cap C(\Om)$ is \emph{\p$_0$-harmonic} in $\Om$ if
for all relatively compact $\Om'\subset\Om$ and for all $v$ such that $u-v \in N^{1,p_0}_0(\Om')$
\[
\int_{\Om'}g_u^{p_0}\,d\mu \leq \int_{\Om'}g_v^{p_0}\,d\mu.
\]

It is known that nonnegative \p$_0$-harmonic functions satisfy
Harnack's inequality and the strong maximum principle, there
are no non-constant nonnegative \p$_0$-harmonic functions on all of $X$,
and \p$_0$-harmonic functions have locally H\"older continuous
representatives. See Kinnunen--Shanmugalingam~\cite{KiSh1}
(see also \cite{BjMa}).

In this section we also assume that $X$ is
\emph{linearly locally connected}: there exists a constant $C\geq 1$ such that
each point $x\in X$ has a neighborhood $U_x$ such that for every
ball $B(x,r)\subset U_x$ and for every pair of points $y,z \in
B(x,2r)\setminus \overline{B}(x,r)$, there exists a path in
$B(x,Cr)\setminus\overline{B}(x,r/C)$ joining the points $y$ and $z$.

The following definition was given by Holopainen and Shanmugalingam
in \cite{HoSha}.

\medskip

\begin{deff}
Let $\Om$ be a relatively compact domain in $X$ and $x_0 \in \Om$.
An extended real-valued function $G=G(\cdot,x_0)$ on $\Om$ is said to be
a \emph{singular function with singularity at $x_0$} if
\medskip
\begin{enumerate}

\item[1.]
$G$ is \p$_0$-harmonic and positive in $\Om\setminus\{x_0\}$,

\item[2.]
$G|_{X\setminus\Om}=0$ and $G \in N^{1,p_0}(X\setminus B(x_0,r))$
for all $r>0$,

\item[3.]
$x_0$ is a singularity, i.e.,
\[
\lim_{x\to x_0} G(x) = \Capc_{p_0}(\{x_0\},\Om)^{1/(1-p_0)},
\]
and $\lim_{x\to x_0} G(x) = \infty$ if $\Capc_{p_0}(\{x_0\},\Om)^{1/(1-p_0)} =0$,

\item[4.]
whenever $0 \leq \alpha < \beta < \sup_{x\in\Om}G(x)$,
\[
C_1(\beta-\alpha)^{1-p_0} \leq
\Capc_{p_0}(\Om^\beta,\Om_\alpha) \leq C_2(\beta-\alpha)^{1-p_0},
\]
where $\Om^\beta = \{x\in\Om:\ G(x)\geq \beta\}$, $\Om_\alpha = \{x\in\Om:\ G(x) > \alpha\}$, and $0< C_1,C_2 <\infty$ are constants depending only on $p_0$.

\end{enumerate}
\end{deff}

\medskip

Note that the singular function is necessarily non-constant,
and continuous on $\Om\setminus\{x_0\}$.

We have the following theorem on the existence.

\medskip

\begin{thm} \label{thm:existence}
Let $\Om$ be a relatively compact domain in $X$, $x_0 \in \Om$,
and $Q(x_0)$ the pointwise dimension at $x_0$. Then
there exists a singular function on $\Om$ with singularity at $x_0$.
Moreover, if $p_0 \leq Q(x_0)$, then every singular function $G$ with
singularity at $x_0$ satisfies the condition
\[
\lim_{x\to x_0}G(x) = \infty.
\]
\end{thm}

\medskip

Essentially, the proof follows from the Harnack inequality on spheres
and Corollary~\ref{corollary}. In particular, it is in
the Harnack inequality
on spheres that $X$ is needed to be linearly locally
connected. See, e.g.,
Bj\"orn et al.~\cite[Lemma 5.3]{BMcSha}). We omit the proof.

\medskip

\begin{remark} \label{remark}
The theorem was first proved by Holopainen and Shanmugalingam in
\cite[Theorem 3.4]{HoSha} under the additional assumption that
the measure on $X$ is $Q$-regular, i.e., for all balls $B(x,r)$
a double inequality
\[
C^{-1}r^Q\leq\mu(B(x,r))\leq Cr^Q
\]
holds. (If $\mu$ is $Q$-regular then $X$ is called an Ahlfors regular
space.) There are, however, many instanses where this is not
satisfied. For example, the weights modifying the Lebesgue measure in
$\R^n$, see \cite{HeKiMa}, or systems of vector fields of H\"ormander type,
see e.g. Capogna et al.~\cite{CaDaGa}, are, in general,
not $Q$-regular for any $Q>0$. In this sense our observation seems to
generalize slightly the results obtained in \cite{HoSha}.
\end{remark}

\end{document}